\documentclass{amsart}
\usepackage{listings}
\usepackage{tikz}
\usepackage{amsmath}
\usetikzlibrary{arrows}
\usepackage{pgfplots}
\usepackage{tikz-3dplot}
\pgfplotsset{width=10cm,compat=1.9}
\newcommand{\R}{\mathbb{R}}
\newcommand{\Z}{\mathbb{Z}}
\newcommand{\N}{\mathbb{N}}
\newcommand{\Q}{\mathbb{Q}}
\newcommand{\C}{\mathbb{C}}

\newtheorem{theorem}{Theorem}[section]
\newtheorem{lemma}[theorem]{Lemma}
\newtheorem{proposition}[theorem]{Proposition}
\newtheorem{corollary}[theorem]{Corollary}
\theoremstyle{definition}
\newtheorem{definition}[theorem]{Definition}

\theoremstyle{remark}
\newtheorem{remark}[theorem]{Remark}
\numberwithin{equation}{section}
\newtheorem{notation}[theorem]{Notation}

\begin{document}

\title[Certain polytopes associated to algebraic integers]{On certain polytopes associated to products of algebraic integer conjugates}
\author{Seda Albayrak}
\address{University of Calgary\\2500 University Drive NW\\Calgary, AB T2N 1N4}
\email{gulizar.albayrak@ucalgary.ca}
\thanks{}

\author{Samprit Ghosh}
\address{University of Calgary\\2500 University Drive NW\\Calgary, AB T2N 1N4}
\email{samprit.ghosh@ucalgary.ca}
\thanks{}

\author{Greg Knapp}
\address{University of Calgary\\2500 University Drive NW\\Calgary, AB T2N 1N4}
\email{greg.knapp@ucalgary.ca}
\thanks{G.~K.~is partially supported by a PIMS postdoctoral fellowship and NSERC grant RGPIN-2019-04844.}

\author{Khoa D. Nguyen}
\address{University of Calgary\\2500 University Drive NW\\Calgary, AB T2N 1N4}
\email{dangkhoa.nguyen@ucalgary.ca}
\thanks{All the authors are partially supported by NSERC grant RGPIN-2018-03770 and CRC tier-2 research stipend 950-231716. They wish to thank Yann Bugeaud for helpful comments.}






\begin{abstract}
Let $d>k$ be positive integers. Motivated by an earlier result of Bugeaud and Nguyen, we let $E_{k,d}$ be the set of $(c_1,\ldots,c_k)\in\mathbb{R}_{\geq 0}^k$ such that $\vert\alpha_0\vert\vert\alpha_1\vert^{c_1}\cdots\vert\alpha_k\vert^{c_k}\geq 1$ for any algebraic integer $\alpha$ of degree $d$, where we label its Galois conjugates as $\alpha_0,\ldots,\alpha_{d-1}$ with
$\vert\alpha_0\vert\geq \vert\alpha_1\vert\geq\cdots \geq \vert\alpha_{d-1}\vert$. First, we give an explicit description of $E_{k,d}$ as a polytope with $2^k$ vertices. Then we prove that for $d>3k$, for every $(c_1,\ldots,c_k)\in E_{k,d}$ and for every $\alpha$ that is not a root of unity, the strict inequality $\vert\alpha_0\vert\vert\alpha_1\vert^{c_1}\cdots\vert\alpha_k\vert^{c_k}>1$
 holds. We also provide a quantitative version of this inequality in terms of $d$ and the height of the minimal polynomial of $\alpha$. 
\end{abstract}

\maketitle

\section{Introduction}
In \cite[Theorem~1.1]{Bugeaud2023}, Bugeaud and Nguyen apply the main theorem of \cite{KMN2019} to prove the following.
\begin{theorem}\label{thm:cFrac}
	Let $\Vert\cdot\Vert$ denote the distance to the nearest integer. Let $\xi$ be an algebraic number of degree $d\geq 3$.  Let $\epsilon > 0$.  Let $(u_n)_{n \geq 1}$ be a non-degenerate linear recurrence sequence of rational integers which is not a polynomial sequence.  Then the set \[\left\{n \in \N : u_n \neq 0 \text{ and } \|u_n\xi\| < \frac{1}{|u_n|^{1/(d-1) + \epsilon}}\right\}\] is finite.
\end{theorem}

This theorem yields a much stronger version of earlier results by Lenstra and Shallit, see \cite{LS1993} and \cite[pp.~20--21]{Bugeaud2023}. It is explained in \cite[p.~20]{Bugeaud2023} that the best possible exponent when $d=3$ is $1/(d-1)=1/2$, and then the authors ask whether the exponent $1/(d-1)$ remains optimal when $d>3$.

The source of the exponent $1/(d-1)$ boils down to the following observation. Let $\alpha$ be an algebraic integer of degree $d\geq 3$ and label its Galois conjugates as $\alpha_0,\alpha_1,\ldots,\alpha_{d-1}$ with $\vert\alpha_0\vert\geq\cdots\geq\vert\alpha_{d-1}\vert$. Then, we have
$\vert\alpha_0\alpha_1^{d-1}\vert\geq\vert\alpha_0\cdots\alpha_{d-1}\vert \geq 1$. 
The question now is whether one can replace $d-1$ by a larger real number $c_1$
so that the inequality $\vert\alpha_0\vert\vert\alpha_1\vert^{c_1}\geq 1$ remains valid for \emph{any} $\alpha$. In this paper, we investigate the more general question in arbitrary dimensions and obtain some related Diophantine inequalities.

\begin{notation}\label{notation}
    Throughout this paper, for an algebraic number $\alpha$ (respectively for a polynomial $f(x)\in\C[x]$) of degree $d\geq 2$, 
    we label the Galois conjugates of $\alpha$
    (respectively the roots of $f(x)$) as $\alpha_0,\ldots,\alpha_{d-1}$ so that 
    $$\vert\alpha_0\vert\geq\vert\alpha_1\vert\geq\cdots\geq\vert\alpha_{d-1}\vert.$$
    While this labelling is not unique (when  $\vert\alpha_i\vert=\vert\alpha_j\vert$ for some $i\neq j$), this causes no ambiguity in the paper.
\end{notation}

\begin{definition}\label{def:Ekd}
    Let $d>k$ be positive integers.  Let $E_{k,d}$ be the set of all tuples $(c_1,\dots,c_k) \in (\R_{\geq 0})^k$ so that for every 
    algebraic integer $\alpha$ of degree $d$, we have
    \begin{align}\label{eq:EkdDefiningIneq}
    \vert\alpha_0\vert\vert\alpha_1\vert^{c_1}\cdots\vert\alpha_k\vert^{c_k} \geq 1.
    \end{align}
\end{definition}

Observe that $E_{k,d}$ is a non-empty closed convex  subset of
the ``first orthant'' $(\R_{\geq 0})^k$. Our first main result
gives an explicit description of $E_{k,d}$ as a polytope in $\R^k$ with exactly $2^k$ vertices:
\begin{theorem}\label{thm:EkdShape}
    Let  $d>k$ be positive integers.  Then, $E_{k,d}$ is a polytope in $\R^k$ with $2^k$ vertices. Moreover, for any $J \subseteq \{1,2,\dots,k\}$, if we write
    $J = \{j_1,\dots,j_n\}$ where $j_1 < j_2 < \dots < j_n$, then we have that $(v_1,\dots,v_k)$ is a vertex of $E_{k,d}$ where \[v_j = \begin{cases} 0 & j \notin J\\ \frac{j_{\ell+1} - j_\ell}{j_1} & j=j_\ell \text{ for some } \ell < n\\ \frac{d-j_n}{j_1} & j = j_n \end{cases}.\]
\end{theorem}

\begin{remark}
    Theorem \ref{thm:EkdShape} shows that the vertices of $E_{k,d}$ are in bijection with the subsets of $\{1,2,\dots,k\}$ where the empty subset corresponds to the vertex $(0,\ldots,0)$. 
\end{remark}

For example, $E_{1,d}$ is the interval $[0,d-1]$ and the exponent $1/(d-1)$
in \cite[Theorem~1.1]{Bugeaud2023} is optimal. The set $E_{2,d}$ is the quadrilateral with vertices $(0,0)$, $(d-1,0)$, $(1,d-2)$, and $(0,(d-2)/2)$. The set $E_{3,d}$ is a polyhedron with $8$ vertices, etc.

\begin{figure}[ht]
   \centering
  \label{fig:shapes}
    \begin{minipage}[b]{0.3\textwidth}
        \centering
        \begin{tikzpicture}[scale=0.5]
            \draw[->,gray] (0,0) -- (5,0) node[below]{$x$};
            \filldraw[black] (0,0) circle (2pt) node[below]{(0,0)};
            \filldraw[black] (3,0) circle (2pt) node[below]{($d-1$,0)};
            \draw[black, very thick] (0,0) -- (3,0);
        \end{tikzpicture}
        \caption{$E_{1,d}$}
    \end{minipage}
  \hspace{-0.07\textwidth}
    \begin{minipage}[b]{0.32\textwidth}
       \centering
        \begin{tikzpicture}[scale=0.55] 
            \draw[->,gray] (0,0) -- (7,0) node[below]{$\, \,x$};
            \draw[->,gray] (0,0) -- (0,5) node[anchor=north east]{$y$};
            \shadedraw (0,0) -- ++(0,2) -- ++(1,2) -- ++(5,-4);
            \filldraw[black] (0,0) circle (2pt) node[below]{(0,0)};
            \filldraw[black] (6,0) circle (2pt);
            \draw (5.5,0) node[below]{($d-1$,0)};
            \draw (1.5,4) node[above]{(1,$d-2$)};
            \filldraw[black] (1,4) circle (2pt);
            \filldraw[black] (0,2) circle (2pt) node[left]{(0,$\frac{d-2}{2}$)};
            \draw[black, very thick] (0,0) -- (6,0);
            \draw[black, very thick] (0,0) -- (0,2);
            \draw[black, very thick] (0,2) -- (1,4);
            \draw[black, very thick] (1,4) -- (6,0);
        \end{tikzpicture}
        \caption{$E_{2,d}$}
    \end{minipage} \hspace{0.04\textwidth}
    \begin{minipage}[b]{0.36\textwidth}
        \centering
        \tdplotsetmaincoords{120}{30}
        \begin{tikzpicture}[tdplot_main_coords, scale=1]
        
            \coordinate (A) at (0, 0, 0);
            \coordinate (B) at (4, 0, 0);
            \coordinate (C) at (0, 1.5, 0);
            \coordinate (D) at (0, 0, 2/3);
            \coordinate (E) at (1, 3, 0);
            \coordinate (F) at (2, 0, 2);
            \coordinate (G) at (0, 0.5, 1);
            \coordinate (H) at (1, 1, 2);

             \shadedraw[shading=axis] (0,0,0) -- ++(0,1.5,0) -- ++(1,1.5,0) -- ++(3,-3,0);

            \draw[black] (A) -- (B) -- (E) -- (C) -- (A);
            \draw[purple] (A) -- (B) -- (F) -- (D);
            \draw[green] (A) -- (C) -- (G) -- (D);
            \draw[black] (B) -- (H);
            \draw[black] (C) -- (H);
            \draw[black] (D) -- (H);
            \draw[black] (G) -- (H);
            \draw[black] (F) -- (H);
            \draw[black] (E) -- (H);
            \draw[black] (A) -- (B);
            \draw[black] (A) -- (D);

        
            \draw[->,gray] (0,0,0) -- (4.5,0,0) node[anchor=north east]{$x$};
            \draw[->,gray] (0,0,0) -- (0,3.5,0) node[anchor=north west]{$y$};
            \draw[->,gray] (0,0,0) -- (0,0,2.5) node[anchor=south]{$z$};

            \filldraw[black] (A) circle (1.5pt) node[below]{$A$};
            \filldraw[black] (B) circle (1.5pt) node[above]{$B$};
            \filldraw[black] (E) circle (1.5pt) node[below]{$E$};
            \filldraw[black] (C) circle (1.5pt) node[below]{$C$};
            \filldraw[black] (D) circle (1.5pt) node[left]{$D$};
            \filldraw[black] (G) circle (1.5pt) node[right]{$G$};
            \filldraw[black] (F) circle (1.5pt) node[above]{$F$};
            \filldraw[black] (H) circle (1.5pt) node[above right]{$\, H$};
        \end{tikzpicture}
        \caption{$E_{3,d}$}
    \end{minipage}
\end{figure}

For the second main result, we first prove that although $E_{k,d}$ is defined by the non-strict inequality \eqref{eq:EkdDefiningIneq}, the strict inequality holds for every $(c_1,\ldots,c_k)\in E_{k,d}$ and every $\alpha$ that is not a root of unity, when $d > 3k$. Then, we give a positive lower bound for 
\begin{equation}\label{eq:product-1}
\vert\alpha_0\vert\vert\alpha_1\vert^{c_1}\cdots\vert\alpha_k\vert^{c_k}-1.
\end{equation}

Finding a lower bound to \eqref{eq:product-1} is also related to several interesting problems in Diophantine approximation. Here, we discuss a couple of them. First, we may assume that $(c_1,\ldots,c_k)$ is a vertex of $E_{k,d}$ by convexity. At the vertex $(c_1,\ldots,c_k)=(0,\ldots,0)$, an optimal lower bound for 
$$\vert\alpha_0\vert\vert\alpha_1\vert^{c_1}\cdots\vert\alpha_k\vert^{c_k}-1=\vert\alpha_0\vert-1$$ is of the form $C/d$, where $C$ is an absolute constant. This was proposed by Schinzel-Zassenhaus in 1965 \cite{SZ1965} and established by Dimitrov's ingenious arguments in 2019 \cite{Dimitrov2019}. For an arbitrary vertex $(c_1,\ldots,c_k)$, one may consider the problem of producing a sharp lower bound for 
the expression \eqref{eq:product-1} as a more general version of the Schinzel-Zassenhaus problem. Unlike the original Schinzel-Zassenhaus in which the lower bound depends only on $d$, for a general $(c_1,\ldots,c_k)$, a lower bound for \eqref{eq:product-1} should involve the height (i.e. the maximum of the absolute values of the coefficients) of the minimal polynomial of $\alpha$, see Remark~\ref{rem:SZ bound can't happen}.  

Since the vertex $(c_1,\ldots,c_k)$ has rational entries, the expression \eqref{eq:product-1} is an algebraic integer. Once it is known to be positive, an easy way to derive a lower bound is to use the product formula: one may take the reciprocal of the product of absolute values of all the other Galois conjugates as a lower bound. We would like to note that this straightforward approach would yield a weaker lower bound than our result.

Instead, we use the description of $E_{k,d}$ in Theorem~\ref{thm:EkdShape} to reduce the current problem to getting a lower bound for $\vert\alpha_k\vert-\vert\alpha_{d-1}\vert$. This is closely related to the ``absolute root separation problem'' (see \cite{BDFPS2022} and the reference therein). We note that the arguments in \cite{BDFPS2022} are more involved and the results are stronger than just using the product formula. 

Before we state our second main theorem, we recall that the height of a polynomial with integer coefficients is the maximum of the absolute values of the coefficients.

\begin{theorem}\label{thm:quantitative}
    Let $k$ and $d$ be positive integers with $d>3k$, $\mu=\displaystyle\left\lceil\frac{\lceil d/3\rceil-k}{2}\right\rceil$ and 
    $$\mathcal{E}=\max\left\{2(d-1)(d-2),\frac{(d-1)(d-2)(d-3)}{2\mu}\right\}.$$
    There exists a positive constant $C$ depending only on $d$ such that the following holds. For every algebraic integer $\alpha$ of degree $d$ that is not a root of unity and for every $(c_1,\ldots,c_k)\in E_{k,d}$, we have
    $$\vert\alpha_0\vert\vert\alpha_1\vert^{c_1}\cdots\vert\alpha_k\vert^{c_k} > 1 + C H^{-\mathcal{E}+1/(k(d-1))},$$
    where $H$ is the height of the minimal polynomial of $\alpha$.
\end{theorem}

\begin{remark}
The ``extra saving term'' $H^{1/(k(d-1))}$ is possible due to some ad hoc arguments that are specific for the  problem that we consider. In any case, it is much smaller than (the reciprocal of) the ``main term'' $H^{-\mathcal{E}}$. Note that except for a few cases, we have $\mathcal{E}=(d-1)(d-2)(d-3)/(2\mu)$. Indeed, observe that
$$\frac{(d-1)(d-2)(d-3)}{2\mu}<2(d-1)(d-2) \text{ implies } d-3<4\left(\frac{(d/3)+1-k}{2}+1\right).$$
The last inequality yields $k=1$ and $4\leq d\leq 20$ or $k=2$ and $7\leq d\leq 14$. In fact, by checking each case, we know exactly when which expression is the maximum:
$$\mathcal{E}= \begin{cases}
  2(d-1)(d-2)  & \text{if}\ (d,k)\in\{(4,1),(5,1),(6,1),(10,1)\},\\
  \displaystyle\frac{(d-1)(d-2)(d-3)}{2\mu} & \text{otherwise}.
\end{cases}$$
\end{remark}

\begin{remark}
    We cannot replace $3k$ in the inequality $d>3k$ by a smaller quantity in order to have $\vert\alpha_0\vert\vert\alpha_1\vert^{c_1}\cdots\vert\alpha_k\vert^{c_k}>1$ for every $(c_1,\ldots,c_k)\in E_{k,d}$
    and for $\alpha$ that is not a root of unity.
    To see why, suppose $d=3k$. Note that $(0,\ldots,0,2)$ is a vertex of $E_{k,3k}$ by taking $J=\{k\}$ in Theorem~\ref{thm:EkdShape}. Consider a cubic field $K\subset\R$ with 2 complex-conjugate embeddings. Let $\xi>1$ be the fundamental unit of $K$ and let $\xi_1$ and $\xi_2$ be the remaining Galois conjugates of $\xi$. Hence $\xi$ is not a $p$-th power of an element of $K$ for every prime $p$, and by considering $\operatorname{Nm}_{K/\Q}(-\xi/4)$, we have that $-\xi/4$ is not a $4$-th power in $K$. By 
    \cite[p.~297]{LangAlg}, $\alpha_0:=\xi^{1/k}$ has degree $3k$. In our notation, $\alpha_0,\ldots,\alpha_{k-1}$ are the $k$-th roots of $\xi$ while $\alpha_k,\ldots,\alpha_{3k-1}$ are the $k$-th roots of $\xi_1$ and $\xi_2$. From 
    $\vert\xi\vert\vert\xi_1\vert^2=\vert\xi\vert\vert\xi_2\vert^2=1$, 
    we have $\vert\alpha_0\vert\vert\alpha_k\vert^2=1$.
\end{remark}

Our result in Theorem \ref{thm:quantitative} is closely related to \cite[Theorem~1]{BDFPS2022} and its proof. In fact, a direct application of \cite[Theorem~1]{BDFPS2022}  would give the weaker version of Theorem~\ref{thm:quantitative} in which $\displaystyle\frac{(d-1)(d-2)(d-3)}{2\mu}$
is replaced by $\displaystyle\frac{(d-1)(d-2)(d-3)}{2}$. Note that when $d-3k\gg d$ (for example, if $d>4k$ then $d-3k>d/4$), then the improvement in Theorem~\ref{thm:quantitative} is significant
since we have $\mathcal{E}=O(d^2)$ and we get a lower bound of the form $1+ C H^{-O(d^2)}$ instead of one of the form $1 + C H^{-O(d^3)}$. This improvement is possible thanks to an adaptation of the arguments in \cite{BDFPS2022} to our current problem. In \cite[Section~4]{BDFPS2022}, the authors comment that the exponent in their lower bounds seems too high. We also believe that the exponent $\mathcal{E}$ in Theorem~\ref{thm:quantitative} is far from optimal.

The organization of this paper is as follows. In Section \ref{sec:auxiliary}, we gather several auxiliary results that are needed for the proofs of our main theorems. These include various estimates involving the roots of special families of polynomials that give rise to the vertices of $E_{k,d}$, a general result confirming that the intersection of certain $2k$ many half-spaces is a polytope of $2^k$ many vertices, and an absolute root separation theorem by Bugeaud, Dujella, Fang, Pejkovi{\'c}, and Salvy \cite{BDFPS2017,BDFPS2022}. Then the proofs of Theorem~\ref{thm:EkdShape} and Theorem~\ref{thm:quantitative} are given in Sections \ref{sec:thm1} and \ref{sec:thm2}, respectively.

\section{Auxiliary Results} \label{sec:auxiliary}
\subsection{Some Useful Families of Polynomials}
We are especially interested in the following family of trinomials.
\begin{definition}\label{fdkhDef}
    For any integers $d,j,h \in \Z$ where $d > j > 0$, let $f_{d,j,h}(x)$ denote \[f_{d,j,h}(x) = x^d - hx^j + 1.\]
\end{definition}

We note that there are infinitely many integers $h$ so that $f_{d,j,h}(x)$ is irreducible in $\Z[x]$. This follows directly from the Hilbert Irreducibility Theorem: see \cite[Theorem 4]{VGR} and take $f(x,y)=f_{d,j,y}(x) = x^d - yx^j + 1$, which is irreducible in $\mathbb{Z}[x,y]$. We now turn our attention to the sizes of the roots of $f_{d,j,h}(x)$.

\begin{proposition}\label{prop:RootAnnulus}
    For any integers $d,j$, and $h$ with $d > j > 0$ and $|h| \geq  3$, and for any real $\epsilon$ satisfying \[\frac{1}{|h|} < \epsilon < 1 - \frac{1}{|h|},\] the polynomials $f_{d,j,h}(x)$ has $d-j$ roots in the annulus \[\left((1-\epsilon)|h| \right)^{\frac{1}{d-j}} < |z| < \left((1+\epsilon)|h| \right)^{\frac{1}{d-j}}.\]
    and $j$ roots in the annulus \[\left((1+\epsilon)|h| \right)^{-\frac{1}{j}} < |z| < \left((1-\epsilon)|h| \right)^{-\frac{1}{j}}.\]
\end{proposition}
\begin{proof}
     We first analyze the larger annulus.  We will apply Rouch\'e's theorem to  $g(x) = x^d - hx^j$, $\ell(x) = 1$, and the circle of radius $R = \left((1+\epsilon)|h| \right)^{\frac{1}{d-j}}$ centered at the origin. On this circle, we have
    \begin{align*}
        |g(x)|  &\geq |x^{j}(x^{d-j} - h)| \\
                  &\geq \big((1+\epsilon)|h| \big)^{\frac{j}{d-j}} \cdot ((1+\epsilon)|h| - |h|)\\
                &= \epsilon (1+\epsilon)^{\frac{j}{d-j}} |h|^{\frac{d}{d-j}} \; > 1.
    \end{align*}
    
    The last inequality follows from the choice of $\epsilon  > |h|^{-1} > |h|^{ - \frac{d}{d-j}}$. Thus, all roots $\alpha$ of $f_{d, j, h}(x)$ satisfy $|\alpha| < \big((1+\epsilon)|h| \big)^{\frac{1}{d-j}}.$
    
    For the smaller circle of radius $r =\left((1-\epsilon)|h| \right)^{\frac{1}{d-j}}$, we apply Rouch\'e's theorem to the polynomials $u(x) = x^d + 1$ and $v(x) = -hx^j$. On this circle, we have
    \begin{align*}
        |v(x)| &= |h|\big((1-\epsilon)|h| \big)^{\frac{j}{d-j}} = |h|^{\frac{d}{d-j}} (1-\epsilon)^{\frac{j}{d-j}} \\
        |u(x)| &\leq |x|^d + 1 = \big((1-\epsilon)|h| \big)^{\frac{d}{d-j}} +1.
    \end{align*}
    Thus,
    \begin{align*}
        \left|\frac{u(x)}{v(x)}\right| &\leq \frac{\big((1-\epsilon)|h| \big)^{\frac{d}{d-j}} +1}{|h|^{\frac{d}{d-j}} (1-\epsilon)^{\frac{j}{d-j}}} = 1 - \epsilon + \frac{1}{|h|^{\frac{d}{d-j}}(1-\epsilon)^{\frac{j}{d-j}}}.
    \end{align*}
    
    As a consequence, we will be able to conclude that $|u(x)| < |v(x)|$ if we can show that $|h|^{-d} < \epsilon^{d-j}(1-\epsilon)^j$. The real function $m(t) = t^{d-j}(1-t)^j$ has a unique critical point in $(0,1)$, and moreover, that critical point is a maximum.  Hence, for any closed interval $I \subset [0,1]$, the function $m(t)$ is minimized at one of the endpoints of $I$.  Since $\epsilon \in \left(\frac{1}{|h|},1-\frac{1}{|h|}\right)$, we have that
    \begin{align*}
        \epsilon^{d-j}(1 - \epsilon)^j &> \min\left(\left(\frac{1}{|h|}\right)^{d-j}\left(1-\frac{1}{|h|}\right)^j, \left(1 - \frac{1}{|h|}\right)^{d-j}\left(\frac{1}{|h|}^{j}\right)\right) \\
        &= \min\left(\frac{(|h| - 1)^{j}}{|h|^d},\frac{(|h|-1)^{d-j}}{|h|^d}\right) > \frac{1}{|h|^d}.        
    \end{align*}
     Therefore $|v(x)| > |u(x)|$ on the circle of radius $r$. By Rouch\'e's Theorem, $f_{d,j,h}(x)$ 
    has exactly $j$ roots inside the circle (and none on the circle). Hence
    $f_{d,j,h}$ has exactly $d-j$ roots in the annulus $r<|z|<R$.

    For the smaller annulus, observe that the reciprocal polynomial of $f_{d,j,h}(x)$ (which we denote by $f_{d,j,h}^*(x)$) is equal to $f_{d,d-j,h}(x)$.  By applying the first half of our argument to $f_{d,j,h}^*(x) = f_{d,d-j,h}(x)$, we find that $f_{d,j,h}^*(x)$ has $j$ roots satisfying \[\left((1-\epsilon)|h| \right)^{\frac{1}{j}} < |z| < \left((1+\epsilon)|h| \right)^{\frac{1}{j}}.\]  But this exactly implies that $f_{d,j,h}(x)$ has $j$ roots in the annulus \[\left((1+\epsilon)|h| \right)^{-\frac{1}{j}} < |z| < \left((1-\epsilon)|h| \right)^{-\frac{1}{j}}.\]
\end{proof}

Choosing $\epsilon = \frac{1}{2}$ in Proposition \ref{prop:RootAnnulus} gives us the following corollary.
\begin{corollary}\label{cor:fdkhAnnuli}
    For any integers $d,j,$ and $h$ with $d > j > 0$ and $|h| \geq 3$, the polynomial $f_{d,j,h}(x)$ has $j$ roots in the annulus \[\left(\frac{2}{3|h|}\right)^{1/j} < |z| < \left(\frac{2}{|h|}\right)^{1/j}\] and $d-j$ roots in the annulus \[\left(\frac{|h|}{2}\right)^{1/(d-j)} < |z| < \left(\frac{3|h|}{2}\right)^{1/(d-j)}.\]
\end{corollary}

\begin{remark}\label{rem:SZ bound can't happen}
    Fix $d\geq 2$, let $j=d-1$, and write $f_h(x)=x^d-hx^{d-1}+1$ where $h$ is an integer such that $|h|\geq 3$ and $f_h$ is irreducible over $\Q$. Let $r_0,\ldots,r_{d-1}$ be roots of 
    $f_h$ as in Notation~\ref{notation}. In Proposition~\ref{prop:RootAnnulus}, take $\epsilon=1.1/|h|$ so that
    $$|r_0|<(1+\epsilon)|h|\ \text{and}\ |r_1|<\left((1-\epsilon)|h|\right)^{-1/(d-1)}.$$
    Therefore,
    $$|r_0||r_1|^{d-1}<\frac{1 + \epsilon}{1-\epsilon} = 1 + \frac{2.2}{|h| - 1.1}.$$
     
    Consequently, we see a contrast between expected lower bounds for $|\alpha_0|$ and $|\alpha_0||\alpha_1|^{d-1}$.  When $\alpha$ is an algebraic integer of degree $d$ which is not a root of unity, the Schinzel-Zassenhaus conjecture (proved by Dimitrov) predicts a lower bound for $|\alpha_0|$ of the form $1 + C/d$ where $C$ is an absolute constant.  However, this example shows that no similar lower bound will hold for $|\alpha_0||\alpha_1|^{d-1}$.  Instead, a lower bound on $|\alpha_0||\alpha_1|^{d-1}$ must involve the height of the minimal polynomial of $\alpha$.
\end{remark}

\subsection{The Vertices of Certain Polytopes}
  For a more detailed introduction, see \cite[Chapters~2, 3]{BG2003}.
A half-space in $\mathbb{R}^k$ is a subset defined by a single linear inequality $$a_1x_1 + a_2x_2 + \cdots + a_kx_k \geq \ell$$ for some $a_1,\dots,a_k,\ell \in \R$ where the $a_i$'s are not all $0$.  An intersection of finitely many half-spaces is called a polyhedral set. A polytope is a bounded polyhedral set. Equivalently, a polytope is the convex hull of finitely many points.  A point $p$ of a polyhedral set $S$ is called an extreme point if whenever 
$p=\lambda p_1+(1-\lambda)p_2$ for some $p_1,p_2\in S$ and $\lambda\in (0,1)$ then $p=p_1=p_2$.
A polytope has finitely many extreme points which are called vertices, and it is the convex hull of its vertices. The following description of the vertices is well-known, but we include a simple proof for the sake of completeness.

\begin{lemma}\label{lem:all vertices}
    Let $I$ be a non-empty set. For $i\in I$, let $L_i$ be a non-zero real linear form in the indeterminates $x_1,\ldots,x_k$, and let $\ell_i\in\mathbb{R}$. Suppose the intersection $P$
    of the half-spaces
    $$L_i(x_1,\ldots,x_k)\geq \ell_i\ \text{for $i\in I$}$$
    is a polytope. Let $J\subseteq I$ with $\vert J\vert=k$ such that the linear system
    $$L_j(x_1,\ldots,x_k)=\ell_j\ \text{for $j\in J$}$$
    has a unique solution $(c_1,\ldots,c_k)$ which belongs to $P$. Then $(c_1,\ldots,c_k)$ is a vertex of $P$. Moreover, every vertex of $P$ arises this way.
\end{lemma}
\begin{proof}
    Write $\mathbf{c}=(c_1,\ldots,c_k)$. If it is not a vertex then we have distinct points $p_1,p_2\in P$ and $\lambda\in (0,1)$ such that $\mathbf{c}=\lambda p_1+(1-\lambda)p_2$. For $j\in J$, from
    $L_j(\mathbf{c})=\ell_j$, $L_j(p_1)\geq \ell_j$, and $L_j(p_2)\geq \ell_j$, we must have
    $L_j(p_1)=L_j(p_2)=\ell_j$. This violates the assumption that the linear system has a unique solution. Therefore $\mathbf{c}$ is a vertex.

    Now suppose that $\mathbf{v}$ is a vertex of $P$. Let $J'=\{i\in I:\ L_i(\mathbf{v})=\ell_i\}$, hence $L_i(\mathbf{v})>\ell_i$ for $i\in I\setminus J'$. If $\mathbf{v}$ is not the unique solution of the linear system
    \begin{equation}\label{eq:for bfv}
    L_j(x_1,\ldots,x_k)=\ell_j\ \text{for $j\in J'$}
    \end{equation}
    then the solution set contains a line passing through $\mathbf{v}$. Hence there is a sufficiently small line segment $S$ such that $\mathbf{v}\in S$, $\mathbf{v}$ is not an endpoint of $S$, $L_j(s)=\ell_j$ for $s\in S$ and $j\in J'$, and $L_i(s)>\ell_i$ for $s\in S$ and $i\in I\setminus J'$. Therefore $S\subseteq P$ and hence $\mathbf{v}$ cannot be a vertex, contradiction. 
    We have proved that $\mathbf{v}$ is the unique solution of \eqref{eq:for bfv}. Then we simply 
    find $J''\subseteq J'$ with $\vert J''\vert=k$ such that $\mathbf{v}$ is the unique solution of the system $L_{j}=\ell_j$ for $j\in J''$. This finishes the proof.
\end{proof}

This lemma indicates that each vertex of a polytope in $\R^k$ comes from a set of $k$ equations of the form $\{L_i(x_1,\dots,x_k) = \ell_i\}$.  However, given only the inequalities $L_i(x_1,\dots,x_k) \geq \ell_i$, it is not clear which size-$k$ sets of equations $\{L_i(x_1,\dots,x_k) = \ell_i\}$ will yield vertices and which will not.  This is called the vertex enumeration problem, which is a fundamental problem in combinatorics and computer science. While there are many algorithms to solve this, it is proved to be NP-hard \cite{KBBEG2008} (however the authors remark that their proof involves unbounded polyhedral sets and hence the hardness of the vertex enumeration problem for polytopes remains open). Therefore, it is interesting that we can describe all the vertices in the below special families of polytopes that, as we will show, include the set $E_{k,d}$.

\begin{proposition}\label{prop:verticesAbstract}
    Let $a_1,\ldots,a_k,b_1,\ldots,b_k$ be positive numbers such that $a_ib_j-a_jb_i>0$ for $1\leq i<j\leq k$. Then the intersection $P$ of the following $2k$ half-spaces:    
    \begin{equation}\label{eq:2k with ai bi}
    \begin{aligned}
         a_1 - b_1 x_1  - \cdots - b_1 x_k & \geq 0 \\
    		 a_2 + a_2 x_1 - b_2 x_2 - \cdots - b_2 x_k & \geq 0 \\ 
    		\vdots \\
    		 a_k + a_k x_1 + \cdots + a_k x_{k-1} - b_k x_k & \geq 0 \\
        x_1 \geq 0, \; x_2 \geq 0, \; \cdots, \; x_k &\geq 0 \\   
    \end{aligned}
    \end{equation}
    is a $k$-dimensional polytope with $2^k$ many vertices that can be described as follows. 
    
    For $1\leq i\leq k$, let $L_i$ denote the left-hand side of the $i$-th inequality in \eqref{eq:2k with ai bi}. For each subset $I$ of $\{1,\ldots,k\}$, the linear system
    $$L_i(x_1,\ldots,x_k)=0\ \text{for $i\in I$ and}\ x_j=0\ \text{for $j\in\{1,\ldots,k\}\setminus I$}$$
    has a unique solution that is a vertex of $P$. When $I$ runs over the $2^k$ many subsets of 
    $\{1,\ldots,k\}$, we obtain all the $2^k$ many vertices of $P$.
\end{proposition}
\begin{proof}
    The last $k$ inequalities together with the first inequality imply that $P$ is bounded. For a sufficiently small $\epsilon>0$ depending on the $a_i$'s and $b_i$'s, we have that
    $P$ contains the hypercube $[0,\epsilon]^k$. Therefore $P$ is a $k$-dimensional polytope.
    We prove the remaining assertions by induction on $k$. The case $k=1$ is obvious since $P$ is given by $a_1-b_1x_1\geq 0$ and $x_1\geq 0$. We consider $k\geq 2$ and assume that the proposition holds for smaller values of $k$.

    \textbf{Claim:} If $1\leq i\leq k$, then there does not exist $(c_1,\ldots,c_k)\in P$ such that $L_i(c_1,\ldots,c_k)=0$ and $c_i=0$. 

    To prove this claim, we suppose otherwise and arrive at a contradiction. If $i=k$ then we have
    $$c_k=\frac{a_k}{b_k}\left(1+c_1+\cdots+c_{k-1}\right)\ \text{and}\ c_k=0,$$
    contradiction. We consider $1\leq i\leq k-1$. From
    $$L_i(c_1,\ldots,c_k)=0,\ c_i=0,\ \text{and}\ L_{i+1}(c_1,\ldots,c_k)\geq 0,$$
    we have:
    $$\frac{a_i}{b_i}\left(1+c_1+\cdots+c_{i-1}\right)=c_{i+1}+\cdots+c_k\leq \frac{a_{i+1}}{b_{i+1}}\left(1+c_1+\cdots+c_{i-1}\right).$$
    This yields a contradiction since $a_i/b_i>a_{i+1}/b_{i+1}$. We finish proving the proposition.

    The above claim implies the following two statements:
    \begin{itemize}
        \item[(i)] If $I,J\subseteq\{1,\ldots,k\}$ such that $\vert I\vert+\vert J\vert=k$ and $I\cap J\neq \emptyset$ then the linear system
        $$L_i(x_1,\ldots,x_k)=0\ \text{for $i\in I$ and}\ x_j=0\ \text{for $j\in J$}$$
        does not have a solution in $P$.

        \item[(ii)] If $I_1$ and $I_2$ are distinct subsets of $\{1,\ldots,k\}$ then the linear system
        $$L_i(x_1,\ldots,x_k)=0\ \text{for $i\in I_1$ and}\ x_j=0\ \text{for $j\in\{1,\ldots,k\}\setminus I_1$}$$
        and the linear system
        $$L_i(x_1,\ldots,x_k)=0\ \text{for $i\in I_2$ and}\ x_j=0\ \text{for $j\in\{1,\ldots,k\}\setminus I_2$}$$
        do not have a common solution in $P$. 
    \end{itemize}

    By using Lemma~\ref{lem:all vertices}, it remains to show that for each subset $I$ of $\{1,\ldots,k\}$, the linear system
    \begin{equation}\label{eq:lin sys in proof}
    L_i(x_1,\ldots,x_k)=0\ \text{for $i\in I$ and}\ x_j=0\ \text{for $j\in \{1,\ldots,k\}\setminus I$}
    \end{equation}
    has a unique solution and this solution is in $P$.

    First, consider the case $k\notin I$, then the above system has the equation $x_k=0$. Then the inequality $L_k(x_1,\ldots,x_k)\geq 0$ becomes
    $a_k+a_kx_1+\cdots+a_kx_{k-1}\geq 0$ which holds automatically as long as the remaining $2k-2$ inequalities hold.
    Therefore we reduce to the case of the polytope in $\mathbb{R}^{k-1}$ 
    given by
    $$L_i(x_1,\ldots,x_{k-1},0)\geq 0\ \text{and}\ x_i\geq 0\ \text{for $1\leq i\leq k-1$}$$
    and apply the induction hypothesis.
    
    Now we assume $k\in I$, hence we have the equation
    $$a_k+a_kx_1+\cdots+a_kx_{k-1}-b_kx_k=0$$
    in \eqref{eq:lin sys in proof}. By substituting $x_k=\displaystyle\frac{a_k}{b_k}\left(1+x_1+\cdots+x_{k-1}\right)$
    and noting that the inequality $x_k\geq 0$ holds automatically as long as the remaining $2k-2$ inequalities hold, we arrive at the polytope $P'$ in $\mathbb{R}^{k-1}$ given by the intersection of $2k-2$ many half-spaces:
    \begin{equation}\label{eq:2k-2 with ai' bi'}
\begin{aligned}
     a_1' - b_1' x_1  - \cdots - b_1' x_{k-1} & \geq 0 \\
		 a_2' + a_2' x_1 - b_2' x_2 - \cdots - b_2' x_{k-1} & \geq 0 \\ 
		\vdots \\
		 a_{k-1}' + a_{k-1}' x_1 + \cdots + a_{k-1}'x_{k-2}- b_{k-1}' x_{k-1} & \geq 0 \\
    x_1 \ge 0, \; x_2 \ge 0 , \; \cdots \;, \; x_{k-1} &\ge 0 \\
\end{aligned}
\end{equation}
where $a_i'=a_i-b_i\cdot\frac{a_k}{b_k}$ and $b_i'=b_i\left(1+\frac{a_k}{b_k}\right)$ for $1\leq i\leq k-1$. For $1\leq i\leq k-1$, we denote the left-hand side of the $i$-th inequality in \eqref{eq:2k-2 with ai' bi'} by $L_i'(x_1,\ldots,x_{k-1})$.

From the conditions on the $a_i$'s and $b_i$'s, we have $a_i'>0$ and $b_i'>0$ for $1\leq i\leq k-1$. Moreover, for $1\leq i<j\leq k-1$, we have:
$$a_i'b_j'-a_j'b_i'=(a_ib_j-a_jb_i)\left(1+\frac{a_k}{b_k}\right)>0.$$
By the induction hypothesis, the linear system:
$$L_i'(x_1,\ldots,x_{k-1})=0\ \text{for $i\in I\setminus\{k\}$ and}\ x_j=0\  \text{for $j\in\{1,\ldots,k-1\}\setminus I$}$$
has a unique solution $(s_1,\ldots,s_{k-1})$ that is in $P'$. Put $s_k=\frac{a_k}{b_k}\left(1+s_1+\cdots+s_{k-1}\right)$, then $(s_1,\ldots,s_k)$ is the unique solution
of \eqref{eq:lin sys in proof} and it is in $P$. This finishes the proof.
\end{proof}

\subsection{Absolute root separation}
Let $P(x)\in \Z[x]$ with degree $d$, height $H$, and roots $r_0,\ldots,r_{d-1}$. Deriving a lower bound for
$$\min\{\vert\vert r_i\vert-\vert r_j\vert\vert:\ \vert r_i\vert\neq \vert r_j\vert\}$$
in terms of $H$ and $d$ is called the absolute root separation problem. This is considered in \cite{GS1996,DS2015,BDFPS2017,Dubickas2019,Sha2019,BDFPS2022}. In particular, we have the following result of Bugeaud, Dujella, Fang, Pejkovi\'c, and Salvy \cite{BDFPS2017,BDFPS2022}: 
\begin{theorem}\label{thm:BDFPS}
    Let $P(x)\in\Z[x]$ be  a polynomial of degree $d$ and let $\alpha$ and $\beta$ be two of its roots such that $\vert\alpha\vert\neq\vert\beta\vert$, then
    \begin{itemize}
        \item [(1)] if $\alpha$ and $\beta$ are real, then $\vert\vert\alpha\vert-\vert\beta\vert\vert\gg_d H^{-(d-1)}$;
        \item [(2)] if $\alpha$ is real and $\beta$ is not, then $\vert\vert\alpha\vert-\vert\beta\vert\vert\gg_d H^{-2(d-1)(d-2)}$;
        \item [(3)] if neither of them is real, then $\vert\vert\alpha\vert-\vert\beta\vert\vert\gg_d H^{-(d-1)(d-2)(d-3)/2}$. 
    \end{itemize}
\end{theorem}

\begin{notation}
In the theorem above, we used the Vinogradov symbols $\gg_d$ and $\ll_d$ to mean that the implied constants depend \textit{only} on $d$.
\end{notation}


\section{Proof of Theorem~\ref{thm:EkdShape}} \label{sec:thm1}
To prove Theorem~\ref{thm:EkdShape}, we use the families of polynomials $f_{d,j,h}(x) = x^d - hx^j + 1$ to derive a useful collection of linear inequalities which every point in $E_{k,d}$ must satisfy.  Then we apply Proposition~\ref{prop:verticesAbstract} to those linear inequalities to find that $E_{k,d}$ must be contained in a polytope.  To show the reverse containment, we identify the vertices of that polytope, and show that each vertex belongs to $E_{k,d}$.

%
%
%
%

\begin{proposition}\label{prop:EkdIneqs}
    Suppose that $(c_1,\dots,c_k) \in E_{k,d}$.  Then for any $1 \leq j \leq k$, we have \[\frac{d-j}{j}+\frac{d-j}{j}c_1 + \cdots + \frac{d-j}{j}c_{j-1} - c_j - \dots - c_k \geq 0.\]
\end{proposition}

\begin{proof}
    Since $(c_1,\dots,c_k) \in E_{k,d}$, we know that for any monic, irreducible $f(x) \in \Z[x]$ with degree $d$ and roots $\alpha_0,\dots,\alpha_{d-1}$ as in Notation~\ref{notation}, \[|\alpha_0||\alpha_1|^{c_1}\cdots|\alpha_k|^{c_k} \geq 1.\]

    Fix any $j$ with $1 \leq j \leq k$, and consider $h \in \Z$ for which $f_{d,d-j,h}(x)$ is irreducible over $\Q$.  The roots $\alpha_0^{(h)},\dots,\alpha_{d-1}^{(h)}$ of $f_{d,d-j,h}(x)$ satisfy \[|\alpha_0^{(h)}|,\dots,|\alpha_{j-1}^{(h)}| < \left(\frac{3|h|}{2}\right)^{1/j}\] and \[|\alpha_j^{(h)}|,\dots,|\alpha_{d-1}^{(h)}| < \left(\frac{2}{|h|}\right)^{1/(d-j)}\] by Corollary \ref{cor:fdkhAnnuli}.  Therefore, for infinitely many $h \in \Z$, we have
    \begin{align*}
        0   &\leq \log|\alpha_0^{(h)}| + c_1\log|\alpha_1^{(h)}| + \cdots + c_k \log|          \alpha_k^{(h)}|\\
            &\leq \left(\frac{1}{j} + \frac{c_1}{j} + \cdots + \frac{c_{j-1}}{j}\right)\log\left(\frac{3|h|}{2}\right) + \left(\frac{c_j}{d-j} + \cdots + \frac{c_k}{d-j}\right)\log\left(\frac{2}{|h|}\right)\\
            &= \left(\left(\frac{1}{j} + \frac{c_1}{j} + \cdots + \frac{c_{j-1}}{j}\right)\log(3/2) + \left(\frac{c_j}{d-j} + \cdots + \frac{c_k}{d-j}\right)\log2  \right)+\\
            &\qquad\qquad +\left(\frac{1}{j} + \frac{c_1}{j} + \cdots + \frac{c_{j-1}}{j} - \frac{c_j}{d-j} - \cdots - \frac{c_k}{d-j}\right)\log|h|.
    \end{align*}

    The only way that this last quantity is nonnegative for infinitely many values of $h \in \Z$ is if \[\frac{1}{j} + \frac{c_1}{j} + \cdots + \frac{c_{j-1}}{j} - \frac{c_j}{d-j} - \cdots - \frac{c_k}{d-j} \geq 0.\] 
\end{proof}

Now we will use the linear inequalities we just derived to inspire the definition of the following set, which we will show to be a polytope that is equal to $E_{k,d}$.
\begin{definition}\label{def:Pkd}
    Let $P_{k,d} \subset \R^k$ denote the set of $(x_1,\dots,x_k)$ satisfying  the following system of linear inequalities:
    \begin{align*}
        \frac{d-j}{j}+\frac{d-j}{j}\sum_{i=1}^{j-1}x_i-\sum_{i=j}^k x_i   &\geq 0 &\text{for each } 1 \leq j \leq k\\
        x_j &\geq 0 &\text{for each } 1 \leq j \leq k.
    \end{align*}
\end{definition}

\begin{proposition}\label{prop:Pkdisinthestatement}
    $P_{k,d}$ is a $k$-dimensional polytope in $\R^k$ with exactly $2^k$ vertices. Moreover, these vertices are in bijection with the subsets of $\{1,\ldots,k\}$ as follows. If $J\subseteq \{1,\ldots,k\}$ and $J=\{j_1,\ldots,j_n\}$ with $j_1<j_2<\cdots<j_n$, then each vertex $(v_1,\ldots,v_k)$ of $P_{k,d}$ is given by
    \[v_j = \begin{cases} 0 & j \notin J\\ \frac{j_{\ell+1} - j_\ell}{j_1} & j=j_\ell \text{ for some } \ell < n\\ \frac{d-j_n}{j_1} & j = j_n \end{cases}.\]
\end{proposition}
\begin{proof}
    From Proposition~\ref{prop:verticesAbstract}, it remains to show that $(v_1,\ldots,v_k)$
    is a solution to the linear system:
    \begin{align}
        \frac{d-j}{j}+\frac{d-j}{j}\sum_{i=1}^{j-1}x_i-\sum_{i=j}^k x_i&=0 & j \in J\label{eq:nonzeroCoordinate}\\
        x_j &= 0 & j \notin J\label{eq:zeroCoordinate}.
    \end{align}
    We only need to verify that $(v_1,\ldots,v_k)$ satisfy the equations in \eqref{eq:nonzeroCoordinate}. Let $j\in J$, then 
    \[\left(-\frac{d-j}{j}\sum_{i=1}^{j-1}v_i\right) + \sum_{i=j}^k v_i = -\frac{d-j}{j} \cdot \frac{j-j_1}{j_1} + \frac{d-j}{j_1} = \frac{d-j}{j}.\]
\end{proof}

\begin{proof}[Proof of Theorem \ref{thm:EkdShape}.]
    Proposition~\ref{prop:EkdIneqs} implies that $E_{k,d}\subseteq P_{k,d}$ and Proposition~\ref{prop:Pkdisinthestatement} implies that $P_{k,d}$ is the polytope in the statement of Theorem~\ref{thm:EkdShape}. It remains to show $P_{k,d}\subseteq E_{k,d}$. 
    By convexity, we only need to prove that the vertices of $P_{k,d}$ are in $E_{k,d}$.

    Let $J$ be any subset of $\{1,\ldots,k\}$ and write $J=\{j_1,\ldots,j_n\}$ with $j_1<\cdots<j_n$. Let $(v_1,\ldots,v_k)$ be the vertex corresponding to $J$ as in the statement of Proposition~\ref{prop:Pkdisinthestatement}. If $J = \emptyset$, then $(v_1,\dots,v_k) = (0,\dots,0)$, yielding \[|\alpha_0||\alpha_1|^{v_1}\cdots|\alpha_k|^{v_k} = |\alpha_0| \geq 1.\] Otherwise, we have
    \begin{align}\label{eq:vertexInequality}
        \begin{split}
        \left(|\alpha_0||\alpha_1|^{v_1}\cdots|\alpha_k|^{v_k}\right)^{j_1}&= |\alpha_0|^{j_1}|\alpha_{j_1}|^{j_2-j_1}|\alpha_{j_2}|^{j_3-j_2}\cdots|\alpha_{j_n}|^{d-j_n}\\
        &\geq |\alpha_0||\alpha_1|\cdots|\alpha_{d-1}| \geq 1.
        \end{split}
    \end{align}
\end{proof}

\section{Proof of Theorem~\ref{thm:quantitative}} \label{sec:thm2}
\begin{lemma}\label{lem:<d/3}
    Let $\alpha$ be an algebraic unit of degree $d \geq 2$ such that $\alpha$ is not a root of unity. 
    Then for $0\leq i<d/3$, we have $\vert\alpha_i\vert>\vert\alpha_{d-1}\vert$.
\end{lemma}
\begin{proof}
    Let $i^*$ be the smallest index such that $\vert\alpha_{i^*}\vert=\vert\alpha_{d-1}\vert$. Assume that $i^*<d/3$ and we will arrive at a contradiction.  Let $a=\vert\alpha_{i^*}\vert=\ldots=\vert\alpha_{d-1}\vert$ which is
    the minimum value among the moduli of the $\alpha_i$'s. Since $\alpha$ is a unit that is not a root of unity, we have $\vert\alpha_0\vert>1>a$.

    Let $\sigma$ be a Galois automorphism such that 
    $\sigma(\alpha_{i^*})=\alpha_0$. Applying $\sigma$ to
    $$\alpha_{i^*}\overline{\alpha_{i^*}}=\ldots=\alpha_{d-1}\overline{\alpha_{d-1}},$$
    we get
    $$\alpha_0\sigma(\overline{\alpha_{i^*}})=\ldots=\sigma(\alpha_{d-1})\sigma(\overline{\alpha_{d-1}}).$$
    Since $\vert\alpha_0\vert>a$ and $\vert\sigma(\overline{\alpha_{i^*}})\vert\geq a$, we have
    $\vert\alpha_0\sigma(\overline{\alpha_{i^*}})\vert>a^2$.

    The preimage $\sigma^{-1}(\{\alpha_1,\ldots,\alpha_{i^*-1}\})$ has $i^*-1$ many elements. Therefore 
    the set
    $$S:=\{i^*+1\leq j\leq d-1:\ \sigma(\alpha_j)\notin\{\alpha_1,\ldots,\alpha_{i^*-1}\}\}$$
    has at least $d-2i^*$ many elements. From the definition of $S$ and the fact that $\sigma(\alpha_{i^*})=\alpha_0$, we have:
    $$\sigma(\alpha_j)\notin\{\alpha_0,\alpha_1,\ldots,\alpha_{i^*-1}\}\ \text{for every $j\in S$.}$$
    Similarly, since the preimage $\sigma^{-1}(\{\alpha_0,\alpha_1,\ldots,\alpha_{i^*-1}\})$
    has $i^*$ many elements, the set
    $$S':=\{j\in S:\ \sigma(\overline{\alpha_j})\notin\{\alpha_0,\alpha_1,\ldots,\alpha_{i^*-1}\}\}$$
    has at least $d-3i^*$ many elements.

    Since $d-3i^*>0$, the set $S'$ is non-empty. Let $j\in S'$, then we have
    $\sigma(\alpha_j),\sigma(\overline{\alpha_j})\notin\{\alpha_0,\ldots,\alpha_{i^*-1}\}$. Therefore
    $\vert\sigma(\alpha_j)\sigma(\overline{\alpha_j})\vert=a^2$ contradicting the properties that
    $\alpha_0\sigma(\overline{\alpha_{i^*}})=\sigma(\alpha_j)\sigma(\overline{\alpha_j})$
    and $\vert\alpha_0\sigma(\overline{\alpha_{i^*}})\vert>a^2$.
\end{proof}

\begin{lemma}\label{lemma:adapt (1) and (2)}
    Let $\alpha$ be an algebraic unit of degree $d\geq 3$ that is not a root of unity and let $H$ denote the height of the minimal polynomial of $\alpha$. Suppose 
    there exists $0< m< d-1$ such that
    $$\vert\alpha_m\vert\leq 1\ \text{and}\ 0<\vert\alpha_m\vert-\vert\alpha_{d-1}\vert<\vert\alpha_0\vert^{-1/(d-1)}.$$
    Then, 
    \begin{itemize}
        \item [(a)] $\vert\alpha_0\vert\gg_d H^{1/m}$
        \item [(b)] $\vert\alpha_m\vert <2\vert\alpha_0\vert^{-1/(d-1)}\ll_d H^{-1/(m(d-1))}$.
        \item [(c)] If there exist $i, j \in \{m,m+1,\ldots,d-1\}$ such that
        $\vert\alpha_i\vert\neq\vert\alpha_j\vert$ and if $\alpha_{i}, \alpha_j \in \mathbb{R}$, then 
        $$\vert\alpha_m\vert-\vert\alpha_{d-1}\vert\gg_d H^{-(d-1)}.$$
        \item [(d)] If there exists $i \in \{m,m+1,\ldots,d-1\}$ such that $\alpha_i \in \mathbb{R}$, then 
        $$\vert\alpha_m\vert-\vert\alpha_{d-1}\vert\gg_d H^{-2(d-1)(d-2)+1/(m(d-1))}.$$
    \end{itemize}
\end{lemma}
\begin{proof}
    Let $a_0+a_1x+\cdots+a_{d-1}x^{d-1}+x^d$ be the minimal polynomial of $\alpha$. Since $\vert\alpha_i\vert\leq 1$ for $i\geq m$, we have $\vert\alpha_{i_1}\vert\cdots\vert\alpha_{i_j}\vert\leq \vert\alpha_0\vert^m$ for any distinct $i_1,\ldots,i_j\in\{0,\ldots,d-1\}$. Hence $\vert a_i\vert\ll \vert\alpha_0\vert^m$ for $0\leq i\leq d-1$ and this proves part (a).

    For part (b), suppose $\vert\alpha_m\vert\geq 2\vert\alpha_{0}\vert^{-1/(d-1)}$. Since $\vert\alpha_m\vert-\vert\alpha_{d-1}\vert<\vert\alpha_0\vert^{-1/(d-1)}$, we have $\vert\alpha_{d-1}\vert>\vert\alpha_0\vert^{-1/(d-1)}$. But, this implies $1=\vert\alpha_0\vert\cdots\vert\alpha_{d-1}\vert\geq \vert\alpha_0\alpha_{d-1}^{d-1}\vert>1$, which is a contradiction.

    Part (c) follows from $\vert\alpha_m\vert-\vert\alpha_{d-1}\vert\geq \vert\alpha_i\vert-\vert\alpha_j\vert$ and part (1) of Theorem~\ref{thm:BDFPS}.

    For part (d), note that since $\vert\alpha_m\vert\neq \vert\alpha_{d-1}\vert$, there always exists $j\in\{m,\ldots,d-1\}$ such that $\vert\alpha_j\vert\neq\vert\alpha_i\vert$. Among such $j$'s, if there exists a real $\alpha_j$, we can apply part (c) to get a stronger lower bound. Otherwise, by the \emph{proof} of Theorem~\ref{thm:BDFPS}(2) in 
    \cite[p.~808]{BDFPS2022}, we have
    $$\vert\vert\alpha_i\vert^2-\vert\alpha_j\vert^2\vert\gg_d H^{-2(d-1)(d-2)}.$$
    Combining this with part (b), we get
    $$\vert\alpha_m\vert-\vert\alpha_{d-1}\vert\geq \left\vert \vert\alpha_i\vert-\vert\alpha_j\vert \right\vert=\frac{\left\vert\vert\alpha_i\vert^2-\vert\alpha_j\vert^2\right\vert}{\vert\alpha_i\vert+\vert\alpha_j\vert}\gg_d H^{-2(d-1)(d-2)+1/(m(d-1))}.$$
\end{proof}

\begin{proof}[Proof of Theorem~\ref{thm:quantitative}]
    Throughout the proof, the positive constants $C_1,C_2,\ldots, C_5$ depend only on $d$. 
    By convexity of $E_{k,d}$, it suffices to consider the case when $(c_1,\ldots,c_k)$ is a vertex. The first observation is that
    there exists $C_1>1$ such that $\vert \alpha_0\vert\geq C_1$, hence we may assume $(c_1,\ldots,c_k)\neq (0,\ldots,0)$. By Theorem \ref{thm:EkdShape}, we have a nonempty $I=\{i_1,\ldots,i_n\} \subseteq\{1,\ldots,k\}$ with $i_1 < \cdots < i_n$ such that
    \begin{align*}
        c_j = \begin{cases}
            0 & j \notin I\\
            \frac{i_{\ell+1} - i_\ell}{i_1} & j = i_\ell \text{ for some } \ell < n\\
            \frac{d - i_n}{i_1} & j = i_n
        \end{cases}.
    \end{align*}
    Then,
    \begin{align*}
    \left(\vert \alpha_0\vert\vert\alpha_1\vert^{c_1}\cdots\vert\alpha_k\vert^{c_k}\right)^{i_1}
    &=\vert\alpha_0\vert^{i_1}\vert\alpha_{i_1}\vert^{i_2-i_1}\cdots\vert\alpha_{i_n}\vert^{d-i_n}\\
    &\geq \vert\alpha_{0}\vert\cdots\vert\alpha_{i_n-1}\vert\vert\alpha_{i_n}\vert^2\vert\alpha_{i_n+1}\vert\cdots\vert\alpha_{d-2}\vert
    \end{align*}
    and it suffices to prove the desired lower bound for the last quantity. We may assume $\vert\alpha_0\cdots\alpha_{d-1}\vert=1$; otherwise, the above quantity is at least $\vert\alpha_0\cdots\alpha_{d-1}\vert\geq 2$. We may also assume $\vert\alpha_{i_n}\vert<1$; otherwise, the above quantity is at least $\vert\alpha_0\vert\geq C_1$.  We have
    \begin{align}\label{eq:|alphain|-|alphad-1|}
    \begin{split}
    \vert\alpha_{0}\vert\cdots\vert\alpha_{i_n-1}\vert\vert\alpha_{i_n}\vert^2\vert\alpha_{i_n+1}\vert\cdots\vert\alpha_{d-2}\vert-1&=\vert\alpha_0\cdots\alpha_{d-2}\vert(\vert\alpha_{i_n}\vert-\vert\alpha_{d-1}\vert)\\
    &>\vert\alpha_{i_n}\vert-\vert\alpha_{d-1}\vert
    \end{split}
    \end{align}
    We may assume that $\displaystyle\vert\alpha_{i_n}\vert-\vert\alpha_{d-1}\vert<\vert\alpha_0\vert^{-1/(d-1)}$. Otherwise, we have a lower bound of the form
    $$1+\frac{1}{\vert\alpha_0\vert^{1/(d-1)}}\geq 1+\frac{1}{(H+1)^{1/(d-1)}}$$
    which is much better than the one in the statement of Theorem~\ref{thm:quantitative}; here the inequality
    $1/\vert\alpha_0\vert\geq 1/(H+1)$ follows from applying the Cauchy bound \cite{BDFPS2022} for the reciprocal polynomial of the minimal polynomial of $\alpha_0$. By Lemma~\ref{lemma:adapt (1) and (2)}(b), we have
    \begin{equation}\label{eq:<<H^{-1/md-1}}
    \vert\alpha_{i_n}\vert\ll_d H^{-1/(i_n(d-1))}\ll_d H^{-1/(k(d-1))}.
    \end{equation}
    
    Let $k_1$ be the smallest index such that $\vert\alpha_{k_1}\vert=\vert\alpha_{i_n}\vert$
    and let $k_2$ be the largest index such that $\vert\alpha_{k_2}\vert>\vert\alpha_{d-1}\vert$. By Lemma~\ref{lem:<d/3} and the assumption that $3k < d$, we have
    $$k_1\leq i_n\leq k\leq \lceil d/3\rceil -1\leq k_2.$$

    By Lemma~\ref{lemma:adapt (1) and (2)}(c)(d) (with $m=k_1$), we may assume that $\alpha_i$ is not real for $k_1\leq i\leq d-1$. Note that $k_2-k_1+1$ is even since the elements of $\{\alpha_{k_1},\alpha_{k_1+1},\ldots,\alpha_{k_2}\}$ are pairs of complex-conjugate numbers. By relabelling the $\alpha_{k}$'s if necessary, we may assume that the $(\alpha_{k_1},\alpha_{k_1+1}),\ldots,(\alpha_{k_2-1},\alpha_{k_2})$ are pairs of complex-conjugate numbers.
    The number of such pairs is
    $$M:=\frac{k_2-k_1+1}{2}\geq \left\lceil\frac{\lceil d/3\rceil-k}{2}\right\rceil.$$
    To finish the proof, our goal is to prove that 
    $$\vert\alpha_{k_1}\vert-\vert\alpha_{d-1}\vert\gg_d H^{-(d-1)(d-2)(d-3)/(2M)}.$$
    We do this by adapting the method in \cite[p.808]{BDFPS2022}. As in \cite{BDFPS2022}, we let 
    $$\mathcal{S}=\{(i,j,s,t):\ 0\leq i,j,s,t\leq d-1,\ i<j,\ s<t,\ \{i,j\}\cap\{s,t\}=\emptyset\}$$ and consider the polynomial
    $$P(x):=\prod_{(i,j,s,t)\in\mathcal{S}} \left(x^{1/2}-(\alpha_i\alpha_j-\alpha_s\alpha_t)\right).$$
    Let $\delta\geq 0$ be the largest integer such that $x^{\delta}$ divides $P(x)$ and let $Q(x)=P(x)/x^{\delta}$. Then the  $(\vert\alpha_{i}\vert^2-\vert\alpha_{d-1}\vert^2)^2$ for $i=k_1,k_1+2,\ldots,k_2-1$
    are $M$ many positive real roots of $Q(x)$, counted with multiplicity. 
    Write
    $$Q(x)=q_0+q_1x+\cdots+q_Dx^D.$$
    Then, we have the crude estimate $M\leq D\leq d^4$. Note that $q_0\neq 0$. By \cite[pp.~807--808]{BDFPS2022}, we have
    $$\max\vert q_i\vert=H(Q)=H(P)\leq C_2 H^{(d-1)(d-2)(d-3)}.$$

    \textbf{Case 1:} Suppose $\vert q_j\vert\leq H^{j(d-1)(d-2)(d-3)/M}$ for $0\leq j\leq M-1$.
    Let $C_3<1/2$ be a small positive constant that will be specified shortly. We prove that every root $r$ of $Q(x)$ satisfies:
    $$\vert r\vert\geq C_3H^{-(d-1)(d-2)(d-3)/M}.$$
    Assume otherwise that $\vert r\vert<C_3H^{-(d-1)(d-2)(d-3)/M}\leq 1/2$. Then, we have
    \begin{align*}
            1\leq \vert q_0\vert&\leq \vert q_1r\vert+\cdots+\vert q_{M-1}r^{M-1}\vert+\vert q_Mr^M\vert+\cdots+\vert q_Dr^D\vert\\
            &\leq C_3+C_3^2+\cdots+C_3^{M-1}+C_2H^{(d-1)(d-2)(d-3)}\frac{\vert r\vert^M}{1-\vert r\vert}\\
            &\leq MC_3+2C_2C_3^M.
    \end{align*}
    Therefore, if we choose a sufficiently small $C_3$ so that $MC_3+2C_2C_3^M<1$, then we arrive at a contradiction. In this case---recall that $(\vert\alpha_{k_1}\vert^2-\vert\alpha_{d-1}\vert^2)^2$ is a root of $Q(x)$---we have
    $$(\vert\alpha_{k_1}\vert^2-\vert\alpha_{d-1}\vert^2)^2\geq C_3H^{-(d-1)(d-2)(d-3)/M}.$$

    \textbf{Case 2:} Suppose there exists $0\leq \ell\leq M-1$ such that
    $\vert q_{\ell}\vert>H^{\ell(d-1)(d-2)(d-3)/M}$ with $\ell$ maximal. Consider the derivative $Q^{(\ell)}(x)=:\ell!q_{\ell}+q_1'x+\cdots+q_{D-\ell}'x^{D-\ell}$ of order $\ell$. Then, there exists $C_4 > 0$ such that
    $\vert q_j'\vert\leq C_4H^{(j+\ell)(d-1)(d-2)(d-3)/M}$ for $1\leq j\leq M-1-\ell$ and  
    $\vert q_j'\vert\leq C_4H^{(d-1)(d-2)(d-3)}$ for $M-\ell\leq j\leq D-\ell$. We prove that every root $r'$ of $Q^{(\ell)}(x)$ satisfies
    $$\vert r'\vert\geq C_5H^{-(d-1)(d-2)(d-3)/M},$$
    where $C_5<1/2$ is a  small positive constant that will be specified shortly. Assume otherwise that $\vert r'\vert< C_5H^{-(d-1)(d-2)(d-3)/M}\leq 1/2$. Then, we have
    \begin{align*}
            \ell!H^{\ell(d-1)(d-2)(d-3)/M}\leq \vert \ell!q_{\ell}\vert\leq& \vert q_1'r'\vert+\cdots+
            \vert q_{M-1-\ell}'r'^{M-1-\ell}\vert\\
            &+\vert q_{M-\ell}'r'^{M-\ell}\vert+\cdots+\vert q_{D-\ell}r'^{D-\ell}\vert\\
            \leq& C_4(C_5+C_5^2+\cdots+C_5^{M-1-\ell})H^{\ell(d-1)(d-2)(d-3)/M}\\
            &+C_4H^{(d-1)(d-2)(d-3)}\frac{\vert r'\vert^{M-\ell}}{1-\vert r'\vert}\\
            \leq&(MC_4C_5+2C_4C_5^{M-\ell})H^{\ell(d-1)(d-2)(d-3)/M}.
    \end{align*}
    Therefore, if we choose a sufficiently small $C_5$ so that $MC_4C_5+2C_4C_5^{M-\ell}<\ell!$, then we arrive at a contradiction. Now, by repeated applications of the Mean Value Theorem, we get that $Q^{(\ell)}(x)$ has a root on the interval $$\left[(\vert\alpha_{k_2}\vert^2-\vert\alpha_{d-1}\vert^2)^2,(\vert\alpha_{k_1}\vert^2-\vert\alpha_{d-1}\vert^2)^2\right].$$ 
    
    Hence, in both cases, we have
    $$(\vert\alpha_{k_1}\vert^2-\vert\alpha_{d-1}\vert^2)^2 \gg_d H^{-(d-1)(d-2)(d-3)/M}.$$
    Combining this with \eqref{eq:<<H^{-1/md-1}}, we get
    $$\vert\alpha_{k_1}\vert-\vert\alpha_{d-1}\vert=\frac{\vert\alpha_{k_1}\vert^2-\vert\alpha_{d-1}\vert^2}{\vert\alpha_{k_1}\vert+\vert\alpha_{d-1}\vert} \gg_d H^{-(d-1)(d-2)(d-3)/(2M)+1/(k(d-1))}.$$
    Since $\displaystyle M\geq \left\lceil\frac{\lceil d/3\rceil-k}{2}\right\rceil$, we get the desired result.
    \end{proof}

\bibliographystyle{amsplain}

\end{document}